\documentclass[12pt,reqno]{article}

\usepackage{amssymb}
\usepackage{amsmath}
\usepackage{mathrsfs}
\usepackage{amsfonts}
\usepackage[colorlinks,linkcolor=red,anchorcolor=green,citecolor=blue]{hyperref}
\usepackage{graphicx,cite,enumerate,tikz,amsthm,amscd,bm,amssymb,changepage,mathrsfs}
\usepackage{indentfirst}
\newtheorem{proposition}{Proposition}[section]
\newtheorem{lemma}[proposition]{Lemma}
\newtheorem{theorem}[proposition]{Theorem}
\newtheorem{corollary}[proposition]{Corollary}
\newtheorem{remark}[proposition]{Remark}

\renewcommand{\theequation}{\thesection.\arabic{equation}}
\renewcommand{\oddsidemargin}{5mm}

\begin{document}
	\parindent 15pt
	\renewcommand{\theequation}{\thesection.\arabic{equation}}
	\renewcommand{\baselinestretch}{1.15}
	\renewcommand{\arraystretch}{1.1}
	\renewcommand{\vec}[1]{\bm{#1}}
	\title{Multiple normalized solutions for a class of dipolar Gross-Pitaveskii equation with a mass subcritical perturbation}
	\author{Yalin Shen$^{a,}$\thanks{Email: 2112202001@stmail.ujs.edu.cn}~~~Yichen Zhang$^{b,}$\thanks{Email: zyc24095@buaa.edu.cn}~~~Thin Van Nguyen$^{c,}$\thanks{Corresponding author Email: Thinmath@gmail.com}\\
		{\small $^a$ School of Mathematical Sciences, Jiangsu University, Zhenjiang 212013, P. R. China}\\
		{\small $^b$ School of Mathematical Sciences, Beihang University, Beijing 100191, P. R. China}\\
		{\small $^c$ Thai Nguyen University of Education, Department of Mathematics, Luong Ngoc Quyen street,}\\
			{\small Thai Nguyen city, Thai Nguyen, VietNam}
	}
	\date{}
	\maketitle
	\begin{abstract}
		In this paper, we study the existence of multiple normalized solutions to the following dipolar Gross-Pitaveskii equation with a mass subcritical perturbation
			\begin{align*}
			\left\{
			\begin{array}{lll}
				-\frac{1}{2}\Delta u+\mu u+V(\varepsilon x)u+\lambda_1|u|^{2}u+\lambda_2(K\ast|u|^{2})u+\lambda_3|u|^{p-2}u=0, \quad&\text{in}\;  \mathbb{R}^{3},\\
				\int_{{\mathbb{R}}^3}|u|^{2}dx=a^{2},
			\end{array}\right.
		\end{align*}
		where $a,\varepsilon>0$, $2<p<\frac{10}{3}$, $\mu \in \mathbb{R}$ denotes the Lagrange multiplier, $\lambda_3<0$, $(\lambda_1,\lambda_2) \in \left\lbrace (\lambda_1,\lambda_2) \in \mathbb{R}^{2}:\lambda_1<\frac{4\pi}{3}\lambda_2\le 0\; \text{or}\; \lambda_1<-\frac{8\pi}{3}\lambda_2\le 0 \right\rbrace$, $V(x)$ is an external potential, $\ast$ stands for the convolution, $K(x)=\frac{1-3cos^{2}\theta (x)}{|x|^{3}}$ and $\theta (x)$ is the angle between the dipole axis determined by $(0,0,1)$ and the vector $x$. Under some assumptions of $V$, we use variational methods to prove that the number of normalized solutions is not less than the number of global minimum points of $V$ if $\varepsilon> 0$ is sufficiently small.  
	\end{abstract}
	
	\section{Introduction}
	\setcounter{section}{1}
	\setcounter{equation}{0}
	In 1995, Cornell et al. achieved the first Bose Einstein condensate in their experiments on dilute alkali gases. Other condensates of different atoms have been produced since this groundbreaking experiment. In recent years, dipole Bose-Einstein condensate, i.e. condensate composed of particles with permanent electric or magnetic dipole moment, has attracted much attention (see e.g.	\cite{Bao-Cai2013,Bao-Cai-Wang 2010}). By using Gross Pitaevskii approximation to describe dipole Bose-Einstein condensate, the following three-dimensional dimensionless dipolar Gross-Pitaveskii equation can be obtained
	\begin{align}\label{eq:main-01 time}
		\begin{array}{lll}
			i\partial_{t}\psi=-\frac{1}{2}\Delta \psi+V(x)\psi+\lambda_1|\psi|^{2}\psi+\lambda_2(K\ast|\psi|^{2})\psi+\lambda_3|\psi|^{p-2}\psi=0, \;\;(t,x) \in \mathbb{R}^{+}\times \mathbb{R}^{3},
		\end{array}
	\end{align}
	where $\psi:[0,+\infty)\times \mathbb{R}^{3} \rightarrow \mathbb{C}$ is the wave function, $i$ denotes the imaginary unit, $\ast$ stands for the convolution, $2<p<6$,  $V(x)$ is an external potential, $\lambda_{i}(i=1,2,3)$ describe the strength of the three nonlinearities, $K(x)=\frac{1-3cos^{2}\theta (x)}{|x|^{3}}$ and $\theta (x)$ is the angle between the dipole axis determined by $(0,0,1)$ and the vector $x$.
 Moreover, \eqref{eq:main-01 time} with $\lambda_3\neq0$ models quantum fluctuations in the condensate (the so-called Lee-Huang-Yang correction), see e.g. \cite{Feng-Cao-Liu2021}. For more physical background, please refer to \cite{Bellazzini-Forcella2019,Bellazzini-Forcella2021,He-Luo2019} and the references therein. 
	
	To our knowledge, rigorous mathematical research on 
	problem \eqref{eq:main-01 time} with initial condition $\psi_{0}(x)=\psi(0,x) \in H^{1}(\mathbb{R}^{3})$ was first conducted by Carles et al. \cite{Carles- Markowich 2008}, where $V(x)=\frac{|x|^{2}}{2}$ and $\lambda_3=0$. They obtained a unique global solution in $\lambda_1 \ge \frac{4 \pi}{3}\lambda_2 \ge 0$ and finite time blow-up may occur in $\lambda_1 <\frac{4 \pi}{3}\lambda_2$. Later, Antonelli and Sparber \cite{Antonelli-Sparber 2011} studied the stationary equation of problem \eqref{eq:main-01 time}, that is,
	\begin{align}\label{eq 1}
		\begin{array}{lll}
			-\frac{1}{2}\Delta u+\mu u+V(x)u+\lambda_1|u|^{2}u+\lambda_2(K\ast|u|^{2})u+\lambda_3|u|^{p-2}u=0,  \quad&\text{in}\;  \mathbb{R}^{3},
		\end{array}
	\end{align}
	where $\mu \in \mathbb{R}$ and $u(x)$ is a time-independent function satisfying $\psi(x,t)=e^{i\mu t}u(x)$.
	They proved the existence of the solutions of \eqref{eq 1} with $\mu >0$, $\lambda_{3}=0$ and $V\equiv 0$ in $\mathbb{R}^{3}$ when $\lambda_1<\frac{4 \pi}{3}\lambda_2$ and $\lambda_2>0$ or $\lambda_1<-\frac{8 \pi}{3}\lambda_2$ and $\lambda_2<0$.
	Based on \cite{Antonelli-Sparber 2011,Carles- Markowich 2008}, the coordinate plane $(\lambda_1,\lambda_2)$ is divided into the stable regime
	\begin{align}\label{wen ding}
		\left\{
		\begin{array}{lll}
			\lambda_1-\frac{4 \pi}{3}\lambda_2 \ge 0,\;\;\; \lambda_2\ge 0,\;\; \text{or}\\
			\lambda_1+\frac{8 \pi}{3}\lambda_2 \ge 0,\;\;\; \lambda_2\le 0
		\end{array}\right.
	\end{align}
	and the unstable regime     
	\begin{align}\label{bu weding}
		\left\{
		\begin{array}{lll}
			\lambda_1-\frac{4 \pi}{3}\lambda_2 < 0,\;\;\; \lambda_2> 0,\;\; \text{or}\\
			\lambda_1+\frac{8 \pi}{3}\lambda_2 < 0,\;\;\; \lambda_2< 0,
		\end{array}\right.
	\end{align} 
which affect the symbol (see \cite{Dinh2021}) of the nonlinear potential energy 
 \begin{align}\label{N}
	\begin{split}
		N(u):=\int_{{\mathbb{R}}^3}\lambda_1 |u|^{4}+\lambda_2 (K\ast |u|^{2})|u|^{2}dx. 
	\end{split}
\end{align}

In recent years, problem \eqref{eq 1} with perscribed mass $\int_{{\mathbb{R}}^3}|u|^{2}dx=a^{2}$ and $a>0$ has been considered in the stable or unstable regimes, see \cite{Carles-Hajaiej2015,Bellazzini-Jeanjean 2017,Luo-Stylianou 2021,Luo-Stylianou 2021-2,Luo-Yang 2023,Wu-Tang 2024}.               
 For $V(x)=\frac{|x|^{2}}{2}$, $\lambda_3=0$ and $(\lambda_1,\lambda_2) \neq 0$ satisfying \eqref{wen ding},  Carles and Hajaiej \cite{Carles-Hajaiej2015} obtained a non-negative minimizer, which is unique and Steiner symmetric. For $V(x)=\frac{\gamma|x|^{2}}{2}$, $(\lambda_1,\lambda_2)$ satisfying \eqref{bu weding}, $\lambda_{3}=0$ and $\gamma \ge 0$, Bellazzini and Jeanjean \cite{Bellazzini-Jeanjean 2017} proved the existence of the ground state.
  For $V \equiv 0$ in $\mathbb{R}^{3}$ and $(\lambda_1,\lambda_2)$ satisfying \eqref{bu weding}, Luo and Stylianou studied the existence of the ground state when $\lambda_3<0$ and $p=5$, see \cite{Luo-Stylianou 2021} or when $\lambda_3>0$ and $p \in (4,6]$, see \cite{Luo-Stylianou 2021-2}.
  For $V \equiv 0$ in $\mathbb{R}^{3}$, $(\lambda_1,\lambda_2)$ satisfying \eqref{bu weding}, $\lambda_3<0$ and $2<p<\frac{10}{3}$, Luo and Yang \cite{Luo-Yang 2023} studied the existence of the ground state. Moreover, they proved that any ground state is a local minimizer and it is stable under the Cauchy flow. Very recently, for $V \equiv 0$ in $\mathbb{R}^{3}$, $2<p<\frac{10}{3}$ and $(\lambda_1,\lambda_2)$ satisfying \eqref{bu weding}, Wu and Tang \cite{Wu-Tang 2024} considered the existence of multiple normalized solutions by replacing $\lambda_{3}$ with $-h(\epsilon x)<0$, where $\epsilon>0$ and $h$ satisfies some technical assumptions.

  In light of the above discussion, we focus on the existence of multiple normalized solutions of problem \eqref{eq 1} with prescribed
  mass when $V\not \equiv0$ in $\mathbb{R}^{3}$, which has not been studied in the existing literature. For this purpose, in present paper, we consider the following dipolar Gross-Pitaveskii equation
\begin{align}\label{eq:main-01}
	\left\{
	\begin{array}{lll}
		-\frac{1}{2}\Delta u+\mu u+V(\varepsilon x)u+\lambda_1|u|^{2}u+\lambda_2(K\ast|u|^{2})u+\lambda_3|u|^{p-2}u=0, \quad&\text{in}\;  \mathbb{R}^{3},\\
		\int_{{\mathbb{R}}^3}|u|^{2}dx=a^{2},
	\end{array}\right.
\end{align}
where $a,\varepsilon>0$, $2<p<\frac{10}{3}$, $(\lambda_1,\lambda_2,\lambda_3)\in \mathbb{R}^{2}\times \mathbb{R}^{-}$ and $\mu \in \mathbb{R}$ denotes the Lagrange multiplier. 
Moreover, the potential $V$ satisfies the following assumptions:\\
$(V_1)$ $V\in L^{\infty}(\mathbb{R}^{3})$, $V(x)\ge 0$ for all $x\in \mathbb{R}^{3}$.\\
$(V_2)$ $V_{\infty}=\lim\limits_{|x|\rightarrow +\infty}V(x)>V_0:=\min\limits_{x \in \mathbb{R}^{3}}V(x)=0$.\\
$(V_3)$ $V^{-1}(\{0\})=\{c_1,c_2,\dots, c_l\}$ with $c_1=0$ and $c_j \neq c_s$ if $j \neq s$.

Solutions of $\eqref{eq:main-01}$ are the critial points of the energy function $J_{\varepsilon}:H^{1}({\mathbb{R}}^3) \rightarrow \mathbb{R}$ constrained to 
\begin{align}\label{Sa}
	S(a):=\{u \in H^{1}(\mathbb{R}^{3}):\int_{{\mathbb{R}}^3}|u|^{2}=a^{2}\},
\end{align}
where
\begin{align*}
	\begin{split}			J_{\varepsilon}(u):=\frac{1}{2}\int_{{\mathbb{R}}^3}|\nabla u|^{2}dx+\int_{{\mathbb{R}}^3}V(\varepsilon x)|u|^{2}dx+\frac{1}{2}N(u)+\frac{2\lambda_3}{p}\int_{{\mathbb{R}}^3}|u|^{p}dx.
	\end{split}
\end{align*}
Here, $N(u)$ is defined in \eqref{N} and $H^{1}(\mathbb{R}^3)$ is the usual Sobolev space endowed with norm
\begin{align*}
	\|u\|_{H^{1}(\mathbb{R}^3)}:=\left( \int_{{\mathbb{R}}^3}(|\nabla u|^{2}+|u|^{2})dx\right)^{\frac{1}{2}},\qquad \forall u \in H^{1}(\mathbb{R}^3).
\end{align*}
Clearly, $J_{\varepsilon} \in C^{1}(H^{1}(\mathbb{R}^{3}),\mathbb{R})$. 
 If $u$ is a critical point of $J_{\varepsilon}|_{S(a)}$, then it follows from the Lagrange multiplier rule that $(u,\mu)$ is a solution of problem $\eqref{eq:main-01}$ for some $\mu \in \mathbb{R}$.

We point out that conditions $(V_1)$-$(V_3)$ were introduced by \cite{Alves2022,Xu-Song 2024} to study the following Schr\"{o}dinger equation
\begin{align}\label{eq:22}
	\left\{
	\begin{array}{lll}
		-\Delta u+V(\varepsilon x)u=\lambda u+g(u), \quad&\text{in}\;  \mathbb{R}^{N},\\
		\int_{{\mathbb{R}}^N}|u|^{2}dx=a^{2},
	\end{array}\right.
\end{align}
where $a,\varepsilon>0$, $V$ denotes the potential, $\lambda \in \mathbb{R}$ is the Lagrange multiplier and $g$ stands for a function. 
 Alves \cite{Alves2022} proved that the numbers of normalized solutions are at least the numbers of global minimum points of $V$ when $\varepsilon$ is small enough and $g$ is a continuous function with $L^2$-subcritical growth. Later, Xu et al. \cite{Xu-Song 2024} extended the result of \cite{Alves2022} to problem \eqref{eq:22} with mixed nonlinearities, where $g(u)=\beta |u|^{q-2}u+|u|^{p-2}u$ and $2<q<2+\frac{4}{N}<p<\frac{2N}{N-2}$. 
 Inspired by \cite{Xu-Song 2024,Wu-Tang 2024}, our main result is as follows.

 \begin{theorem}\label{th1.1}
 	Let $(\lambda_1,\lambda_2) \in \mathcal{B}:=\left\lbrace (\lambda_1,\lambda_2) \in \mathbb{R}^{2}:\lambda_1<\frac{4\pi}{3}\lambda_2\le 0\; \text{or}\; \lambda_1<-\frac{8\pi}{3}\lambda_2\le 0 \right\rbrace$, $\lambda_3<0$ and $2<p<\frac{10}{3}$. Suppose that $(V_1)$-$(V_3)$ hold. Then there exist $\bar{\varepsilon},V_{\ast}>0$ and $\bar{a}>0$, such that problem \eqref{eq:main-01} has at least $l$ couples $(u_{j},\mu_{j})\in H^{1}(\mathbb{R}^{3}) \times \mathbb{R}^{+}$ of weak solutions for $\|V\|_{\infty}<V_{\ast}$, $\varepsilon \in (0,\bar{\varepsilon})$ and $a\in (0,\bar{a}]$ with $\int_{{\mathbb{R}}^3} |u_j|^{2}dx=a^{2}$ and negative energies for $j=1,2,\ldots,l$.
 \end{theorem}
\begin{remark}
	We must assume that $(\lambda_1,\lambda_2) \in \mathcal{B}$ to ensure that $N<0$ (see Lemma \ref{N<}), which is crucial to our problem. In addition, we emphasize that unlike the approach in \cite{Xu-Song 2024} which requires constructing truncation functions, our method is simpler and only relies on a set, see \eqref{set}.
\end{remark}
\begin{remark}
	An example satisfying conditions $(V_1)$-$(V_3)$ is the following function $V: \mathbb{R}^{3} \rightarrow \mathbb{R}$,
	\begin{align*}
		V(x)=1-\frac{1}{1+|x-k_1||x-k_2|\cdots |x-k_l|},
	\end{align*}	
	where $0=|k_1|<|k_2|<\cdots<|k_l|$.
\end{remark}
The paper is organized as follows. In Section \ref{sec:pre}, we give some preliminary results. Section \ref{Existence result} is devoted to proving the existence of normalized solutions of the autonomous case of problem \eqref{eq:main-01}. Finally, we prove Theorem \ref{th1.1} in Section \ref{sec:Multiplicity result}.

\section{Preliminaries}\label{sec:pre}
\setcounter{section}{2}
\setcounter{equation}{0}
In this section, we give some preliminary results. For convenience, one let $\|\cdot\|_{q}:=\|\cdot\|_{L^{q}(\mathbb{R}^{3})}$ for $q\in [1,+\infty]$. We start with the following Gagliardo-Nirenberg inequality, see \cite{Weinstein 1983}.
\begin{lemma}\label{G-N}
	For any $u \in H^{1}(\mathbb{R}^{3})$ and $q \in(2,6)$, there exists a constant $C_q>0$ such that 
	\begin{align*}
		\|u\|_{q} \le C_{q}\|\nabla u\|_{2}^{\gamma_q}\| u\|_{2}^{1-\gamma_q},
	\end{align*}
	where $\gamma_q:=\frac{3(q-2)}{2q}$.
\end{lemma}

\begin{lemma}\cite[Lemma 2.1]{Wu-Tang 2024}\label{N<}
	Let $u \in H^{1}(\mathbb{R}^{3})$, then $|N(u)|\le \Lambda \|u\|^{4}_{4}$, where $\Lambda:=\text{max}\{|\lambda_1-\frac{4\pi}{3}\lambda_2|,|\lambda_1+\frac{8\pi}{3}\lambda_2|\}$. Moreover, if $(\lambda_1,\lambda_2)\in\mathcal{B}$, then $N(u)<0$.
\end{lemma}

\begin{lemma}\cite[Lemma 2.2]{Wu-Tang 2024}\label{B-L}
	If $u_n \rightharpoonup u$ in $H^{1}(\mathbb{R}^{3})$, then up to a subsequence, 
	\begin{align*}
		\int_{{\mathbb{R}}^3}(K\ast|u_n|^{2})|u_n|^{2}dx=\int_{{\mathbb{R}}^3}(K\ast|v_n|^{2})|v_n|^{2}dx+\int_{{\mathbb{R}}^3}(K\ast|u|^{2})|u|^{2}dx+o_n(1),
	\end{align*}
	where $v_n:=u_n-u$ and $o_{n}(1) \rightarrow 0$ as $n \rightarrow +\infty$. In addition, one has
	\begin{align*}
		N(u_n)=N(v_n)+N(u)+o_n(1).
	\end{align*} 
\end{lemma}

\begin{lemma}\label{xingzhi g}
	Let $a>0$, $2<p<\frac{10}{3}$ and $\lambda_3<0$. Then for any $a>0$, the function \begin{align}\label{g(a,r)}
		\begin{split}
			g(a,r):=\frac{1}{2}-\frac{\Lambda}{2}C^{4}_{4}ra-\frac{2|\lambda_3|}{p}C^{p}_{p}a^{p(1-\gamma_p)}r^{p\gamma_p-2}
		\end{split}
	\end{align}
	has a unique global maximum point. Moreover, the maximum value satisfies
\begin{align*}
		\max\limits_{r>0}g(a,r)
		\begin{cases}
		>0 \quad& \text{if}\; a<\bar{a},\\
		=0 \quad& \text{if}\; a=\bar{a},\\
		<0 \quad& \text{if}\; a>\bar{a},\\
		\end{cases}
	\end{align*}
where $\bar{a}$ is defined in \eqref{bar a}.
\end{lemma}
\begin{proof}
	In view of $2<p<\frac{10}{3}$, we have $0<p \gamma_p<2$. Together this with \eqref{g(a,r)} yields $g(a,r)\rightarrow -\infty$ as $r\rightarrow 0^{+}$ and $g(a,r)\rightarrow -\infty$ as $r\rightarrow +\infty$ for any $a>0$. Furthermore, we define $f_{a}(r):=g(a,r)$, then it follows that
	\begin{align*}
		f_a'(r)=\frac{\partial (g(a,r))}{\partial r}=-\frac{\Lambda}{2}C^{4}_{4}a-\frac{2|\lambda_3|}{p}(p\gamma_p-2)C^{p}_{p}a^{p(1-\gamma_p)}r^{p\gamma_p-3}.
	\end{align*}
	By letting $f_a'(r)=0$, one has
	\begin{align*}
		r_a=\left(\frac{4|\lambda_3|(2-p \gamma_p)C^{p}_{p}a^{(1-\gamma_p)p-1}}{\Lambda C^{4}_{4}p} \right)^{\frac{1}{3-p \gamma_p}}.
	\end{align*}
Therefore, $r_a$ is the only global maximum point of $g(a,r)$. Then, we have
\begin{align*}
	\begin{split}
		\max\limits_{r>0}g(a,r)=&g(a,r_a)\\
		=&\frac{1}{2}-\frac{\Lambda}{2}C^{4}_{4}r_aa-\frac{2|\lambda_3|}{p}C^{p}_{p}a^{p(1-\gamma_p)}r_a^{p\gamma_p-2}\\
		=&\frac{1}{2}-\frac{\Lambda}{2}C^{4}_{4}a\left(\frac{4|\lambda_3|(2-p \gamma_p)C^{p}_{p}a^{(1-\gamma_p)p-1}}{\Lambda C^{4}_{4}p} \right)^{\frac{1}{3-p \gamma_p}}\\
		&-\frac{2|\lambda_3|}{p}C^{p}_{p}a^{p(1-\gamma_p)}\left(\frac{4|\lambda_3|(2-p \gamma_p)C^{p}_{p}a^{(1-\gamma_p)p-1}}{\Lambda C^{4}_{4}p} \right)^{\frac{p\gamma_p-2}{3-p \gamma_p}}\\
		&=\frac{1}{2}-\frac{1}{2}(\Lambda C^{4}_{4})^{\frac{p \gamma_p-2}{p \gamma_p-3}}\left(\frac{4|\lambda_3|(2-p \gamma_p) C^{p}_{p}}{p} \right)^{\frac{1}{3-p \gamma_p}}\cdot a^{\frac{4}{3}}\\
		&-2\left( \frac{C^{p}_{p}|\lambda_3|}{p}\right) ^{\frac{1}{3-p\gamma_p}}\left(\frac{4(2-p\gamma_p)}{\Lambda C^{4}_{4}} \right)^{\frac{p\gamma_p-2}{3-p \gamma_p}}\cdot a^{\frac{4}{3}}\\
		=&\frac{1}{2}-\alpha a^{\frac{4}{3}},
	\end{split}
\end{align*}
where
\begin{align*}
	\alpha:=\frac{1}{2}(\Lambda C^{4}_{4})^{\frac{p \gamma_p-2}{p \gamma_p-3}}\left(\frac{4|\lambda_3|(2-p \gamma_p) C^{p}_{p}}{p} \right)^{\frac{1}{3-p \gamma_p}}+2\left( \frac{C^{p}_{p}|\lambda_3|}{p}\right) ^{\frac{1}{3-p\gamma_p}}\left(\frac{4(2-p\gamma_p)}{\Lambda C^{4}_{4}} \right)^{\frac{p\gamma_p-2}{3-p \gamma_p}}.
\end{align*}
Thus, we can take 
\begin{align}\label{bar a}
	\bar{a}:=\left(\frac{1}{2\alpha} \right)^{\frac{3}{4}}>0,
\end{align}
which concludes the proof.
\end{proof}

\begin{lemma}\label{2.4}
		Let $a>0$, $2<p<\frac{10}{3}$, $\lambda_3<0$ and $(a_2,r_2)\in (0, \bar{a})\times (0,+\infty)$ satisfy $g(a_1,r_1)\ge 0$. Then for any $a_2 \in (0,a_1]$, there holds
	\begin{align}\label{3}
		g(a_2,r_2)\ge 0\quad \text{if}\; r_2\in \left[\frac{a_2}{a_1}r_1,r_1 \right].
	\end{align}
\end{lemma}

\begin{proof}
	By \eqref{g(a,r)}, one has $a\mapsto g(a,r)$ is decreasing for $a>0$. Therefore, we have
	\begin{align}\label{1}
		g(a_2,r_1)\ge g(a_1,r_1)\ge0.
	\end{align}
	  Furthermore, in view of $a_2 \in (0,a_1]$, one has
	\begin{align}\label{2}
		\begin{split}
		g(a_2,\frac{a_2}{a_1}r_1)-g(a_1,r_1)=&
		\frac{1}{2}-\frac{\Lambda}{2}C^{4}_{4}\frac{a_2^{2}}{a_1}r_1-\frac{2|\lambda_3|}{p}C^{p}_{p}a^{p(1-\gamma_p)}_{2}\left(\frac{a_2}{a_1}r_1 \right)^{p\gamma_p-2}\\
		&-\left(\frac{1}{2}-\frac{\Lambda}{2}C^{4}_{4}r_1a_1-\frac{2|\lambda_3|}{p}C^{p}_{p}a^{p(1-\gamma_p)}_{1}r_1^{p\gamma_p-2}\right) \\
		=&\frac{\Lambda}{2}C^{4}_{4}a_1r_1\left(1-\frac{a_2^{2}}{a_1^{2}} \right)+\frac{2|\lambda_3|}{p}C^{p}_{p}r_1^{p\gamma_p-2}a_1^{(1-\gamma_p)p}\left[1-\left(\frac{a_2}{a_1} \right)^{p-2}  \right]\\
		\ge&0. 
		\end{split}
	\end{align}
Then it follows from Lemma \ref{xingzhi g} and \eqref{1}-\eqref{2} that \eqref{3} holds, which concludes the proof.
\end{proof}


\section{Existence result for the autonomous problem}\label{Existence result}
\setcounter{section}{3}
\setcounter{equation}{0}
In this section, we establish the existence of normalized solutions of the autonomous case of problem \eqref{eq:main-01}, that is, 
	\begin{align}\label{eq:zizhi}
	\left\{
	\begin{array}{lll}
		-\frac{1}{2}\Delta u+\mu u+\omega u+\lambda_1|u|^{2}u+\lambda_2(K\ast|u|^{2})u+\lambda_3|u|^{p-2}u=0, \quad&\text{in}\;  \mathbb{R}^{3},\\
		\int_{{\mathbb{R}}^3}|u|^{2}dx=a^{2},
	\end{array}\right.
\end{align}
where $a>0$, $\omega \ge 0$, $2<p<\frac{10}{3}$,
$(\lambda_1,\lambda_2) \in\mathcal{B}$, $\lambda_3<0$   and $\mu \in \mathbb{R}$ denotes the Lagrange multiplier. Solutions of problem \eqref{eq:zizhi} are the critical points of
the corresponding energy functional $J_{\omega}:H^{1}(\mathbb{R}^{3})\rightarrow \mathbb{R}$ constrained to \eqref{Sa}, where
\begin{align*}
	\begin{split}			J_{\omega}(u):=\frac{1}{2}\int_{{\mathbb{R}}^3}|\nabla u|^{2}dx+\int_{{\mathbb{R}}^3} \omega|u|^{2}dx+\frac{1}{2}N(u)+\frac{2\lambda_3}{p}\int_{{\mathbb{R}}^3}|u|^{p}dx.
	\end{split}
\end{align*}

Our main result in this section is as follows.
\begin{theorem}\label{th3.1}
	Let $(\lambda_1,\lambda_2) \in \mathcal{B}$, $\lambda_{3}<0$ and $2<p<\frac{10}{3}$. Suppose that there is $V_{\ast}>0$ satisfying $\omega<V_{\ast}$. Then there exist $\bar{a},\bar{r}>0$ such that for any $a \in (0,\bar{a})$, the functional $J_{\omega}$ has an interior local minimizer on 
	\begin{align}\label{set}
		W^{a}_{\bar{r}}:=\{u \in S(a):\|\nabla u\|_{2}<\bar{r}\}.
	\end{align}
	Moreover, equation \eqref{eq:zizhi} has a couple $(u_a,\mu_a)$ solution with $J_{\omega}(u_a)<0$, where $u_a$ is positive and $\mu_a>0$.
\end{theorem}

To prove the above theorem, we divide the proof into the following lemmas.

\begin{lemma}\label{J in S}
 For any $u \in S(a)$, there holds 
\begin{align*}
	J_{\omega}(u) \ge \|\nabla u\|^{2}_{2}g(a,\|\nabla u\|_{2}).
\end{align*}  	

\end{lemma}
\begin{proof}
	For any $u \in S(a)$, it follows from \eqref{g(a,r)}, Lemmas \ref{G-N} and \ref{N<} that
\begin{align*}
	\begin{split}
	J_{\omega}(u)=&\frac{1}{2}\int_{{\mathbb{R}}^3}|\nabla u|^{2}dx+\int_{{\mathbb{R}}^3}\omega|u|^{2}dx+\frac{1}{2}N(u)+\frac{2\lambda_3}{p}\int_{{\mathbb{R}}^3}|u|^{p}dx\\
	\ge&\frac{1}{2}\int_{{\mathbb{R}}^3}|\nabla u|^{2}dx-\frac{\Lambda}{2}\|u\|^{4}_{4}+\frac{2\lambda_3}{p}\int_{{\mathbb{R}}^3}|u|^{p}dx\\
	\ge& \frac{1}{2}\|\nabla u\|^{2}_{2}-\frac{\Lambda}{2}C^{4}_{4}\|\nabla u\|^{3}_{2}a-\frac{2|\lambda_3|}{p}C^{p}_{p}a^{p(1-\gamma_p)}\|\nabla u\|^{p\gamma_p}_{2}\\
	=&\|\nabla u\|^{2}_{2}\left[\frac{1}{2}-\frac{\Lambda}{2}C^{4}_{4}\|\nabla u\|_{2}a-\frac{2|\lambda_3|}{p}C^{p}_{p}a^{p(1-\gamma_p)}\|\nabla u\|^{p\gamma_p-2}_{2} \right]\\
	=&\|\nabla u\|^{2}_{2}g(a,\|\nabla u\|_{2}). 
	\end{split}
\end{align*} 
The proof is finished.
\end{proof}

In the following, we define
\begin{align*}
	W^{a}_{\bar{r}}:=\{u \in S(a):\|\nabla u\|_{2}<\bar{r}\},
\end{align*}
where 
\begin{align}\label{bar r}
	\bar{r}:=r_{\bar{a}}=\left(\frac{4|\lambda_3|(2-p \gamma_p)C^{p}_{p}\bar{a}^{(1-\gamma_p)p-1}}{\Lambda C^{4}_{4}p} \right)^{\frac{1}{3-p \gamma_p}}>0.
\end{align}
Obviously, it follows from Lemmas \ref{xingzhi g} and \ref{J in S} that $J_{\omega}$ is bounded from below in $W^{a}_{\bar{r}}$ for any $a\in(0,\bar{a})$, where $\bar{a}$ is given by \eqref{bar a}. Then one can define the local minimization problem by
\begin{align*}
	\Upsilon_{\omega}(a):=\inf\limits_{u\in W^{a}_{\bar{r}}}J_{\omega}(u) \; \text{for any}\; a\in(0,\bar{a}).
\end{align*}
Moreover, $\Upsilon_{\omega}(a)$ satisfies the following properties.

\begin{lemma}\label{jixiao<0}
Let $2<p<\frac{10}{3}$ and $\lambda_3<0$. For any $a\in (0,\bar{a})$, there exists $V_{\ast}>0$ such that
\begin{align*}
	\Upsilon_{\omega}(a)<0<\inf\limits_{u\in \partial W^{a}_{\bar{r}}}J_{\omega}(u),\quad\text{if}\;\; \omega<V_{\ast}.
\end{align*}
\end{lemma}
\begin{proof}
	Fixed $u\in S(a)$, we set $u_t(x)=t^{\frac{3}{2}}u(tx)$ for all $t>0$.
	Then, we have
	\begin{align*}
		\int_{{\mathbb{R}}^3}|u_{t}|^{2}dx=a^{2},\;\int_{{\mathbb{R}}^3}|\nabla u_{t}|^{2}dx=t^{2}\int_{{\mathbb{R}}^3}|\nabla u|^{2}dx\;\;
		 \text{and}\; \int_{{\mathbb{R}}^3}|u_{t}|^{p}dx=t^{\gamma_p p}\int_{{\mathbb{R}}^3}|u|^{p}dx.
	\end{align*}
Combining these with Lemma \ref{N<}, one has
\begin{align*}
	\begin{split}
		J_{\omega}(u_t)=&\frac{1}{2}\int_{{\mathbb{R}}^3}|\nabla u_t|^{2}dx+\int_{{\mathbb{R}}^3}\omega|u_t|^{2}dx+\frac{1}{2}N(u_t)-\frac{2|\lambda_3|}{p}\int_{{\mathbb{R}}^3}|u_t|^{p}dx\\
		<&\frac{t^{2}}{2}\int_{{\mathbb{R}}^3}|\nabla u|^{2}dx+\omega a^{2}-\frac{2|\lambda_3|}{p}t^{p\gamma_p}\int_{{\mathbb{R}}^3}|u|^{p}dx\\
		=&t^{p\gamma_p}\left[\frac{t^{2-p\gamma_p}}{2}\int_{{\mathbb{R}}^3}|\nabla u|^{2}dx-\frac{2|\lambda_3|}{p}\int_{{\mathbb{R}}^3}|u|^{p}dx \right]+\omega a^{2}.
	\end{split}
\end{align*}
Then by $0<p\gamma_p<2$ and $u\in W^{a}_{\bar{r}}$, there exists $t_0>0$ small enough such that
\begin{align*}
	t^{p\gamma_p}\left[\frac{t^{2-p\gamma_p}}{2}\int_{{\mathbb{R}}^3}|\nabla u|^{2}dx-\frac{2|\lambda_3|}{p}\int_{{\mathbb{R}}^3}|u|^{p}dx \right]:=\kappa<0.
\end{align*}
Furthermore, we let $V_{\ast}:=\frac{-\kappa}{\bar{a}^{2}}$. Then based on the above facts, we have if $\omega<V_{\ast}$,
\begin{align*}
	J_{\omega}(u_t)<\kappa-\kappa=0.
\end{align*}
It follows that $	\Upsilon_{\omega}(a)<0$ if $\omega<V_{\ast}$. On the other hand, if $u\in \partial W^{a}_{\bar{r}}$, one has $u\in S(a)$ and $\|\nabla u\|_{2}=\bar{r}$. Together these with Lemmas \ref{xingzhi g} and \ref{J in S} yield
\begin{align*}
	J_{\omega}(u)\ge \|\nabla u\|^{2}_{2}g(a,\|\nabla u\|_{2})=\bar{r}^{2}g(a,\bar{r})>0 \; \text{for any}\; a\in(0,\bar{a}),
\end{align*}
which implies $\inf\limits_{u\in \partial W^{a}_{\bar{r}}}J_{\omega}(u)>0$ for any $a\in(0,\bar{a})$. We complete the proof.
\end{proof}

From now on, we fix $\omega<V_{\ast}$, which is defined in Lemma \ref{jixiao<0}.
\begin{lemma}\label{lem3.4}
	 $\Upsilon_{\omega}(a)$ is continuous with regard to $a\in (0,\bar{a})$.
\end{lemma}
\begin{proof}
	For any $a\in (0,\bar{a})$, let $\{a_n\}\subset (0,\bar{a})$ satisfy $a_n \rightarrow a$ as $n\rightarrow +\infty$. Then by the definition of $\Upsilon_{\omega}(a_n)$ and Lemma \ref{jixiao<0}, there exists $\{u_n\}\subset W^{a_n}_{\bar{r}}$ satisfying for all $n\in \mathbb{N}^{+}$,
	\begin{align}\label{4}
		J_{\omega}(u_n)-\frac{1}{n}<\Upsilon_{\omega}(a_n)\;\;\text{and}\;\; J_{\omega}(u_n)<0.
	\end{align}
We claim $\Upsilon_{\omega}(a_n)\rightarrow \Upsilon_{\omega}(a)$ as $n\rightarrow +\infty$. Once this claim is true, then it follows from the definition of continuity that $\Upsilon_{\omega}(a)$ is continuous with regard to $a\in (0,\bar{a})$. Next, we use three steps to prove the above claim.

\textbf{Step 1.} Let $w_n:=\frac{a}{a_n}u_n$, then $\{w_n \}\subset W^{a}_{\bar{r}}$.

Indeed, it is clear that $w_n\in S(a)$. Next, we need to verify $\|\nabla w_n\|_{2}<\bar{r}$. When $a\le a_n$, we can easily obtain  $\|\nabla w_n\|_{2}<\bar{r}$. On the other hand, when $a_n<a<\bar{a}$. By Lemma \ref{xingzhi g} and \eqref{bar r}, one has $g(a,\bar{r})>0$ for all $a \in (0,\bar{a})$. Then it follows from Lemma \ref{2.4} that
\begin{align}\label{step 1}
	g(a_n,r)\ge 0\quad \text{if}\; r\in \left[\frac{a_n}{a}\bar{r},\bar{r} \right]. 
\end{align}
Furthermore, combining \eqref{4} and Lemma \ref{J in S}, one has $0>J_{\omega}(u_n)\ge \|\nabla u_n\|^{2}_{2}g(a,\|\nabla u_n\|_{2})$. Together this with \eqref{step 1} gives $\|\nabla u_n\|_{2}<\frac{a_n}{a}\bar{r}$. Therefore, $\|\nabla w_n\|_{2}=\frac{a}{a_n}\|\nabla u_n\|_{2}<\bar{r}$.

\textbf{Step 2.} There holds $\Upsilon_{\omega}(a)\le \Upsilon_{\omega}(a_n)+o_n(1)$.

By Step 1, $\lim\limits_{n\rightarrow \infty}a_n=a$, $\{u_n\}\subset W^{a_n}_{\bar{r}}$, Lemmas \ref{G-N} and \ref{jixiao<0}, we have
\begin{align*}
	\begin{split}
		\Upsilon_{\omega}(a)\le J_{\omega}(w_n)=&J_{\omega}(u_n)+(J_{\omega}(w_n)-J_{\omega}(u_n))\\
		=&J_{\omega}(u_n)+\frac{1}{2}\left[\left( \frac{a}{a_n}\right)^{2}-1  \right]\int_{{\mathbb{R}}^3}|\nabla u_n|^{2}dx+\left[\left( \frac{a}{a_n}\right)^{2}-1  \right]\int_{{\mathbb{R}}^3}\omega| u_n|^{2}dx\\
		&+\frac{1}{2}\left[\left( \frac{a}{a_n}\right)^{4}-1  \right]N(u)\\
		&-\frac{2|\lambda_{3}|}{p}\left[\left(\frac{a}{a_n} \right)^{p}-1  \right]\int_{{\mathbb{R}}^3}|u_n|^{p}dx
	\longrightarrow	 J_{\omega}(u_n)+o_n(1), \;\;\text{as}\; n\rightarrow \infty.
	\end{split}
\end{align*}
 Together this with \eqref{4} yields
\begin{align*}
	\Upsilon_{\omega}(a)\le J_{\omega}(w_n)=J_{\omega}(u_n)+o_n(1)\le \Upsilon_{\omega}(a_n)+o_n(1).
\end{align*}

\textbf{Step 3.}  There holds
$\Upsilon_{\omega}(a)\ge \Upsilon_{\omega}(a_n)+o_n(1)$.

Indeed, let $\zeta_n:=\frac{a_n}{a}z_{n}$ and $\{u_n\}\subset W^{a_n}_{\bar{r}}$ satisfying for all $n\in \mathbb{N}^{+}$,
\begin{align*}
	J_{\omega}(z_n)-\frac{1}{n}<\Upsilon_{\omega}(a)\;\;\text{and}\;\; J_{\omega}(z_n)<0.
\end{align*}
Then as the same proof in Steps 1-2, we can obtain
\begin{align*}
	\Upsilon_{\omega}(a_n)\le J_{\omega}(\zeta_n)=J_{\omega}(z_n)+(J_{\omega}(\zeta_n)-J_{\omega}(z_n))\le \Upsilon_{\omega}(a)+o_n(1).
\end{align*}
 We complete the proof.
\end{proof}

\begin{lemma}\label{budengshi}
Let $0<a_1<a_2<\bar{a}$. There holds 
\begin{align*}
	\Upsilon_{\omega}(a_2)\le \Upsilon_{\omega}(a_1)+\Upsilon_{\omega}(\sqrt{a^{2}_{2}-a_{1}^{2}}).
\end{align*}
Moreover, if $\Upsilon_{\omega}(a_1)$ or $\Upsilon_{\omega}(\sqrt{a^{2}_{2}-a_{1}^{2}})$ is attained, then the above inequality is strict.
\end{lemma}
\begin{proof}
Let $0<a_1<a_2<\bar{a}$ and $v_n=\theta u_n$, where $0<\theta<\frac{a_2}{a_1}$. Then as the same proof in \eqref{4} and Lemma \ref{lem3.4}-Step 1, there exists $\{u_n\}\subset W^{a_1}_{\bar{r}}$ satisfying for all $n\in \mathbb{N}^{+}$,
\begin{align*}
	J_{\omega}(u_n)<\Upsilon_{\omega}(a_1)+\frac{1}{n}\;\;\text{and}\;\; J_{\omega}(u_n)<0
\end{align*}
and $v_n \in W^{\theta a_1}_{\bar{r}}$. Combining these with Lemma \ref{N<} and $2<p<\frac{10}{3}$, we deduce that
\begin{align}\label{5}
	\begin{split}
		\Upsilon_{\omega}(\theta a_1)\le J_{\omega}(v_n)&=\frac{1}{2}\int_{{\mathbb{R}}^3}|\nabla v_n|^{2}dx+\int_{{\mathbb{R}}^3}\omega|v_n|^{2}dx+\frac{1}{2}N(v_n)-\frac{2|\lambda_3|}{p}\int_{{\mathbb{R}}^3}|v_n|^{p}dx\\
		&=\frac{\theta^{2}}{2}\int_{{\mathbb{R}}^3}|\nabla u_n|^{2}dx+\theta^{2}\int_{{\mathbb{R}}^3} \omega|u_n|^{2}dx+\frac{\theta^{4}}{2}N(u_n)-\frac{2|\lambda_3|\theta^{p}}{p}\int_{{\mathbb{R}}^3}|u_n|^{p}dx\\
		&=\theta^{2}J_{\omega}(u_n)+\frac{\theta^{4}-\theta^{2}}{2}N(u_n)+\frac{2|\lambda_3|(\theta^{2}-\theta^{p})}{p}\int_{{\mathbb{R}}^3}|u_n|^{p}dx\\
		&<\theta^{2}J_{\omega}(u_n)\\
		&\le \theta^{2}\Upsilon_{\omega}(a_1)+\theta^{2}\frac{1}{n},
	\end{split}
\end{align}
which implies that
\begin{align}\label{b}
	\Upsilon_{\omega}(\theta a_1)\le \theta^{2} \Upsilon_{\omega}(a_1),
\end{align}
as $n\rightarrow \infty$. Therefore, we have 
\begin{align*}
	\begin{split}
		\Upsilon_{\omega}(a_2)=&\frac{a_1^2}{a_2^2}\Upsilon_{\omega}\left(\frac{a_2}{a_1} a_1 \right)+\frac{a_2^2-a^2_1}{a_2^2}\Upsilon_{\omega}\left(\frac{a_2}{\sqrt{a_2^2-a^2_1}}\sqrt{a_2^2-a^2_1} \right)\\
		\le& \Upsilon_{\omega}(a_1)+\Upsilon_{\omega}\left(\sqrt{a_2^2-a^2_1} \right).
	\end{split}
\end{align*}
If $\Upsilon_{\omega}(a_1)$ is attained, it follows from \eqref{5} that the inequality is strict. 
The proof is finished.
\end{proof}

\begin{lemma}\label{compactness}
	Let $\{u_n\}\subset W^{a}_{\bar{r}}$ be a minimizing sequence with respect to $J_{\omega}$. Then, for some subsequence, either	\\
	(i) $\{u_n\}$ is strongly convergent, or\\
	(ii) there exists $\{y_n\}\subset \mathbb{R}^{3}$ with $|y_n|\rightarrow+\infty$ such that the sequence $v_n(x)=u_n(x+y_n)$ is strongly convergent to a function $v\in W^{a}_{\bar{r}}$ with $J_{\omega}(v)=\Upsilon_{\omega}(a)$.	
\end{lemma}

\begin{proof}
	Let $\{u_n\}\subset W^{a}_{\bar{r}}$ be a minimizing sequence such that
	\begin{align*}
		J_{\omega}(u_n)=\Upsilon_{\omega}(a)+o_{n}(1).
	\end{align*}
It follows that $\{u_n\}$ is bounded in $H^{1}(\mathbb{R}^{3})$. Therefore, there exists $u\in H^{1}(\mathbb{R}^{3})$ such that $u_n \rightharpoonup u$ in $H^{1}(\mathbb{R}^{3})$.

Case 1. $u\not \equiv 0$ in $\mathbb{R}^{3}$. Let $v_n=u_n-u$, 
it follows from Brezis-Lieb Lemma \cite{Willem 1996} and Lemma \ref{B-L} that
\begin{align}\label{11}
	\|u_n\|^{2}_{2}=\|v_n\|^{2}_{2}+\|u\|^{2}_{2}+o_{n}(1), \quad
	\|\nabla u_n\|^{2}_{2}=\|\nabla v_n\|^{2}_{2}+\|\nabla u\|^{2}_{2}+o_{n}(1),
\end{align}
\begin{align}\label{22}
	\|u_n\|^{p}_{p}=\|v_n\|^{p}_{p}+\|u\|^{p}_{p}+o_{n}(1),\quad
	N(u_n)=N(v_n)+N(u)+o_{n}(1).
\end{align}
Then we can see $\|\nabla u\|_{2}<\bar{r}$ and $\|u\|_{2}\le a$.
Next, we claim $\|u\|_{2}=a$. Suppose by contradiction that $\|u\|_{2}=b<a$, it follows that 
 $\|v_n\|_{2}=\sqrt{a^{2}-b^{2}}$. Then by \eqref{11}, \eqref{22} and Lemma \ref{budengshi}, one has
 \begin{align*}
 	\begin{split}
 	\Upsilon_{\omega}(a)+o_{n}(1)=J_{\omega}(u_n)&=J_{\omega}(v_n)+J_{\omega}(u)+o_{n}(1)\\
 	&\ge \Upsilon_{\omega}(\sqrt{a^{2}-b^{2}})+\Upsilon_{\omega}(b)+o_{n}(1)\\
 	&>\Upsilon_{\omega}(a)+o_{n}(1),
 	\end{split}
 \end{align*}
 which is a contradiction. Thus, $\|u\|_{2}=a$. It follows that $v_n \rightarrow 0$ in  $L^{2}(\mathbb{R}^{3})$ and $u\in W^{a}_{\bar{r}}$. Together this with Fatou Lemma gives
\begin{align*}
	\begin{split}
		\Upsilon_{\omega}(a)&=\lim\limits_{n\rightarrow +\infty}J_{\omega}(u_n)\\
		&=\lim\limits_{n\rightarrow +\infty}\left(\frac{1}{2}\int_{{\mathbb{R}}^3}|\nabla u_n|^{2}dx+\int_{{\mathbb{R}}^3}\omega|u_n|^{2}dx+\frac{1}{2}N(u_n)-\frac{2|\lambda_3|}{p}\int_{{\mathbb{R}}^3}|u_n|^{p}dx \right)\\
		&\ge J_{\omega}(u)\ge \Upsilon_{\omega}(a), 
	\end{split}
\end{align*}
which implies that $J_{\omega}(u)=\Upsilon_{\omega}(a)$ and $\lim\limits_{n\rightarrow +\infty}J_{\omega}(u_n)=J_{\omega}(u)$. Therefore, we have $u_n \rightarrow u$ in $H^{1}(\mathbb{R}^{3})$.

Case 2. $u \equiv 0$ in $\mathbb{R}^{3}$.
We claim there are $\eta,R>0$ and $\{y_n\} \in \mathbb{R}^{3}$ such that for all $n$,
\begin{align}\label{33}
	\int_{B_{R}(y_n)}|u_n|^{2}dx\ge \eta>0,
\end{align}
where $B_{R}(y_n):=\{x \in \mathbb{R}^{3}: |x-y_n|<R\}$.
Once this claim is true, then one has $|y_n|\rightarrow +\infty$. Otherwise, there exists $R_n>0$
such that
\begin{align*}
	\int_{B_{R_n}(0)}|u_n|^{2}dx>\int_{B_{R}(y_n)}|u_n|^{2}dx\ge \eta>0.
\end{align*}
However, $u_n\rightharpoonup 0$ in $H^{1}(\mathbb{R}^{3})$, and consequently, $u_n \rightarrow 0$ in $L^{2}(B_{R_n}(0))$, which is a contradiction. In addition, we let $v_n:=u_n(x+y_n)$, it follows that $\|v_n\|_{2}=\|u_n\|_{2}$, $\|\nabla v_n\|_{2}=\|\nabla u_n\|_{2}$ and $J_{\omega}(v_n)=J_{\omega}(u_n)$. Therefore, there exists $v\in H^{1}(\mathbb{R}^{3})$ such that $v_n \rightharpoonup v$ in $ H^{1}(\mathbb{R}^{3})$. Combining this with \eqref{33}, we have
\begin{align*}
	\eta \le \|u_n\|_{L^{2}(B_{R}(y_n))}=\|v_n\|_{L^{2}(B_{R}(y_n))}\le\|v_n-v\|_{L^{2}(B_{R}(y_n))}+\|v\|_{L^{2}(B_{R}(y_n))},
\end{align*}
which implies $v \not \equiv 0$ in $\mathbb{R}^{3}$.
Then as the same proof in Case 1, we have $v_n \rightarrow v$ in $ H^{1}(\mathbb{R}^{3})$.

Next, we prove the above claim. Assume by contradiction that $	\int_{B_{R}(y_n)}|u_n|^{2}dx=0$. Combining this with \cite[Lemma 1.21]{Willem 1996}, we have $u_n \rightarrow 0$ in $L^{p}(\mathbb{R}^{3})$ for $2<p<6$.
Together this with Lemma \ref{N<} gives
\begin{align*}
	\begin{split}
		J_{\omega}(u_n)&=\frac{1}{2}\int_{{\mathbb{R}}^3}|\nabla u_n|^{2}dx+\int_{{\mathbb{R}}^3}\omega|u_n|^{2}dx+\frac{1}{2}N(u_n)-\frac{2|\lambda_3|}{p}\int_{{\mathbb{R}}^3}|u_n|^{p}dx\\
		&\ge \frac{1}{2}\|\nabla u_n\|_{2}^{2}-\frac{\Lambda}{2}\|u_n\|_{4}^{4}+o_n(1)\\
		&=\frac{1}{2}\|\nabla u_n\|_{2}^{2}+o_n(1),
	\end{split}
\end{align*}
which contradicting to $\Upsilon_{\omega}(a)<0$.
We complete the proof.
\end{proof}

\begin{proof}[Proof of Theorem \ref{th3.1}]
By Lemmas \ref{jixiao<0} and \ref{compactness}, there exists $u_{a} \in W^{a}_{\bar{r}} \subset S(a)$ such that
\begin{align*}
	J_{\omega}(u_a)=\inf\limits_{u\in  W^{a}_{\bar{r}}}J_{\omega}(u)=\Upsilon_{\omega}(a)<0\; \text{and}\; J_{\omega}|^{'}_{S(a)}(u)=0.
\end{align*}
Note that $\|\nabla |u_a|\|_{2} \le \|\nabla u_a\|_{2}$. Then, $u_a$ is a nonegative solution. It follows from the strong maximum principle that $u_a>0$ in $\mathbb{R}^{3}$. Furthermore, by Lemmas \ref{N<} and \ref{jixiao<0}, there exists $\mu_{a}\in \mathbb{R}$ such that
\begin{align*}
	\begin{split}
		-\mu_{a} a^{2}&=\frac{1}{2}\int_{{\mathbb{R}}^3}|\nabla u_a|^{2}dx+\int_{{\mathbb{R}}^3}\omega|u_a|^{2}dx+N(u_a)-|\lambda_3|\int_{{\mathbb{R}}^3}|u_a|^{p}dx\\
		&=\Upsilon_{\omega}(a)+\frac{1}{2}N(u_a)-\frac{(p-2)|\lambda_3|}{p}\int_{{\mathbb{R}}^3}|u_a|^{p}dx\\
		&<0.
	\end{split}
\end{align*}
Thus, $\mu_{a}>0$. The proof is finished.
\end{proof}
\begin{corollary}\label{cor 3.8}
	Let $0 \le \omega_1<\omega_2\le V_{\ast}$, then $\Upsilon_{\omega_1}(a)<\Upsilon_{\omega_2}(a)<0$.
\end{corollary}
\begin{proof}
	Let $u_{\mu_2} \in W^{a}_{\bar{r}}$ such that $J_{\omega_2}(u_{\mu_2})=\Upsilon_{\omega_2}(a)$. Combining this with Lemma \ref{jixiao<0}, we have
	\begin{align*}
		\Upsilon_{\omega_1}(a)\le J_{\omega_1}(u_{\mu_2})<J_{w_2}(u_{\mu_2})=\Upsilon_{\omega_2}(a)<0,
	\end{align*}
	which completes the proof.
\end{proof}


\section{Proof of Theorem \ref{th1.1}}\label{sec:Multiplicity result}

\setcounter{section}{4}
\setcounter{equation}{0}

In this section, we prove Theorem \ref{th1.1}. We start with the following definitions. Let \begin{align*}
	J_{\infty}(u):=\frac{1}{2}\int_{{\mathbb{R}}^3}|\nabla u|^{2}dx+\int_{{\mathbb{R}}^3}V_{\infty}|u|^{2}dx+\frac{1}{2}N(u)-\frac{2|\lambda_3|}{p}\int_{{\mathbb{R}}^3}|u|^{p}dx 
\end{align*}
and
\begin{align*}
	J_{0}(u):=\frac{1}{2}\int_{{\mathbb{R}}^3}|\nabla u|^{2}dx+\frac{1}{2}N(u)-\frac{2|\lambda_3|}{p}\int_{{\mathbb{R}}^3}|u|^{p}dx, \quad \text{for all}\; u\in H^{1}(\mathbb{R}^{3}).
\end{align*}
In what following, we assume that $\|V\|_{\infty}<V_{\ast}$, where $V_{\ast}$ is given by Lemma \ref{jixiao<0}. Combining this with Theorem \ref{th3.1},
we can define the following minimization problems
\begin{align*}
	\Upsilon_{0}(a):=\inf\limits_{u\in W^{a}_{\bar{r}}}J_{0}(u),\;\Upsilon_{\infty}(a):=\inf\limits_{u\in W^{a}_{\bar{r}}}J_{\infty}(u),\;\Upsilon_{\varepsilon}(a):=\inf\limits_{u\in W^{a}_{\bar{r}}}J_{\varepsilon}(u) \; \text{for any}\; a\in(0,\bar{a}).
\end{align*}
\begin{lemma}\label{lem 4.1}
	$\limsup\limits_{\varepsilon \rightarrow 0^{+}}\Upsilon_{\varepsilon}(a) \le \Upsilon_{0}(a)$ and there is $\varepsilon_0>0$ such that $\Upsilon_{\varepsilon}(a)<\Upsilon_{\infty}(a)$  for any $\varepsilon \in (0,\varepsilon_{0})$.
\end{lemma}
\begin{proof}
	By Theorem \ref{th3.1}, there is $u_0 \in W^{a}_{\bar{r}}$ such that $J_{0}(u_0)=\Upsilon_{0}(a)$. Then, we have
	\begin{align*}
		\Upsilon_{\varepsilon}(a) \le J_{\varepsilon}(u_0)=\frac{1}{2}\int_{{\mathbb{R}}^3}|\nabla u_0|^{2}dx+\int_{{\mathbb{R}}^3}V(\varepsilon x)|u_0|^{2}dx+\frac{1}{2}N(u_0)-\frac{2|\lambda_3|}{p}\int_{{\mathbb{R}}^3}|u_0|^{p}dx.
	\end{align*}
Letting $\varepsilon \rightarrow 0^{+}$, it follows from the Lebesgue dominated convergence theorem that
\begin{align*}
	\limsup\limits_{\varepsilon \rightarrow 0^{+}}\Upsilon_{\varepsilon}(a) \le  \limsup\limits_{\varepsilon \rightarrow 0^{+}} J_{\varepsilon}(u_0)=J_{0}(u_0)=\Upsilon_{0}(a).
\end{align*}

On the other hand, combining Corollary \ref{cor 3.8} and $V_{0}=0<V_{\infty}$, we have $\Upsilon_{0}(a)<\Upsilon_{\infty}(a)<0$. Therefore, there is $\varepsilon_0>0$ such that $\Upsilon_{\varepsilon}(a)<\Upsilon_{\infty}(a)$  for any $\varepsilon \in (0,\varepsilon_{0})$. We complete the proof.
\end{proof}

\begin{lemma}\label{lem 4.2}
Suppose that $(V_1)$-$(V_3)$ hold. Let $\{u_n\}\subset W^{a}_{\bar{r}}$ satisfy $J_{\varepsilon}(u_n)\rightarrow c <\Upsilon_{\infty}(a)$ as $n\rightarrow \infty$.  If $u_n \rightharpoonup u_{\varepsilon}$ in $H^{1}(\mathbb{R}^{3})$, then $u_{\varepsilon} \neq 0$.
\end{lemma}
\begin{proof}
	Suppose by contradiction that $u_{\varepsilon} = 0$. In view of $(V_1)$-$(V_3)$, for any $\delta>0$, there exists $R>0$ such that
	\begin{align*}
		V(x)\ge V_{\infty}-\delta\;\; \text{for all}\;|x|\ge R.
	\end{align*}
Then, one has
\begin{align*}
	\begin{split}			c+o_n(1)=&J_{\varepsilon}(u_n)=J_{\infty}(u_n)+\int_{{\mathbb{R}}^3}(V(\varepsilon x)-V_{\infty})|u_n|^{2}dx\\
	\ge& J_{\infty}(u_n)+ \int_{B_{R/ \varepsilon}(0)}(V(\varepsilon x)-V_{\infty})|u_n|^{2}dx-\delta\int_{B^{c}_{R/ \varepsilon}(0)}|u_n|^{2}dx.
	\end{split}
\end{align*}
Due to $u_n \rightharpoonup 0$ in $H^{1}(\mathbb{R}^{3})$, then $\{u_n\}$ is bounded in $H^{1}(\mathbb{R}^{3})$ and $u_n \rightarrow 0$ in $L^{2}(B_{R/\varepsilon}(0))$. It follows that
\begin{align}\label{a}
	c+o_n(1)\ge J_{\infty}(u_n)-\delta C+o_n(1)\ge \Upsilon_{\infty}(a)-\delta C+o_n(1),
\end{align}
for some $C>0$. Since $\delta>0$ is arbitrary, we have $c\ge \Upsilon_{\infty}(a)$, which is a contradiction. Therefore, $u_{\varepsilon} \neq 0$. The proof is finished.
\end{proof}

\begin{lemma}\label{lem 4.3}
	Assume that $(V_1)-(V_3)$ hold. For $a \in (0,\bar{a})$, let  $\{u_n\}$ be a $(PS)_c$ sequence for $J_{\varepsilon}$ restricts to $W^{a}_{\bar{r}}$ with $c<\varUpsilon_{\infty}(a)<0$ and $u_n \rightharpoonup u_{\varepsilon}$ in $H^{1}(\mathbb{R}^{3})$. If $u_n \nrightarrow u_{\varepsilon}$ in $H^{1}(\mathbb{R}^{3})$, there exist $\varepsilon_1>0$ and $\tau>0$ independent of $\varepsilon$ such that
	\begin{align*}
		\liminf\limits_{n \rightarrow +\infty}\|u_n-u_{\varepsilon}\|^{2}_{2} \ge \tau.
	\end{align*}
\end{lemma}
\begin{proof}
	Let $\{u_n\}$ be a $(PS)_c$ sequence for $J_{\varepsilon}$ restricts to $W^{a}_{\bar{r}}$, then 
	\begin{align*}
		J_{\varepsilon}(u_n)\rightarrow c\;\;\;\text{and}\;\;\;\|J_{\varepsilon}|'_{W^{a}_{\bar{r}}}(u_n)\|_{H^{-1}(\mathbb{R}^{3})}\rightarrow 0\;\;\; \;\text{as}\;\; n\rightarrow +\infty.
	\end{align*}
We define the functional $\Phi:H^{1}(\mathbb{R}^{3})\rightarrow \mathbb{R}$ by 
\begin{align*}
	\Phi(u):=\frac{1}{2}\int_{{\mathbb{R}}^3}|u|^{2}dx.
\end{align*}
It follows from \cite[Lemma 5.11 and Proposition 5.12]{Willem 1996} that there exists $\{\mu_n\}\subset \mathbb{R}$ satisfying
\begin{align}\label{111}
	\|J'_{\varepsilon}(u_n)+\mu_n\Phi'(u_n)\|_{H^{-1}(\mathbb{R}^{3})}\rightarrow 0\;\; \text{as}\;\; n\rightarrow+\infty.
\end{align}
In view of $u_n \rightharpoonup u_{\varepsilon}$ in $H^{1}(\mathbb{R}^{3})$, one has $\{u_n\}$ is bounded in $H^{1}(\mathbb{R}^{3})$. Combining this with \eqref{111}, we have $\{\mu_n\}$ is bounded. Then there is $\mu_{\varepsilon} \in \mathbb{R}$ such that $\mu_n \rightarrow \mu_{\varepsilon}$ as $n\rightarrow \infty$. Thus, 
\begin{align*}
	J'_{\varepsilon}(u_\varepsilon)+\mu_\varepsilon\Phi'(u_\varepsilon)=0\;\; \text{in}\;\; H^{-1}(\mathbb{R}^{3})\; \text{and}\;	\|J'_{\varepsilon}(v_n)+\mu_\varepsilon\Phi'(v_n)\|_{H^{-1}(\mathbb{R}^{3})}\rightarrow 0\;\; \text{as}\;\; n\rightarrow+\infty,
\end{align*}
where $v_n=u_n-u_{\varepsilon}$. 
It follows that
\begin{align}\label{222}
	\frac{1}{2}\int_{{\mathbb{R}}^3} |\nabla v_n|^{2}dx+\int_{{\mathbb{R}}^3}V(\varepsilon x)|v_n|^{2}dx+N(v_n)-|\lambda_3|\int_{{\mathbb{R}}^3}|v_n|^{p}dx+\mu_{\varepsilon} \int_{{\mathbb{R}}^3}|v_n|^{2}dx=o_n(1).
\end{align}
Furthermore, it follows from \eqref{111} and $N(u_n)<0$ that
\begin{align*}
	\begin{split}
		0>\Upsilon_{\infty}(a)>c&=\lim\limits_{n\rightarrow +\infty}J_{\varepsilon}(u_n)\\
		&=\lim\limits_{n\rightarrow +\infty}[J_{\varepsilon}(u_n)-\left\langle J'_{\varepsilon}(u_n),u_n \right\rangle -\mu_n a^{2}+o_n(1)]\\
		&=\lim\limits_{n\rightarrow +\infty}\left[-\frac{1}{2}N(u_n)+\left( 1-\frac{2}{p}\right)|\lambda_3|\int_{{\mathbb{R}}^3}|u_n|^{p}dx-\mu_n a^{2}+o_n(1) \right] \\
		&\ge -\mu_\varepsilon a^{2},
	\end{split}
\end{align*}
which implies that
\begin{align*}
	\limsup\limits_{\varepsilon \rightarrow 0^{+}}\mu_{\varepsilon}\ge -\frac{\Upsilon_{\infty}(a)}{a^{2}}.
\end{align*}
Therefore, there exists $\varepsilon_1>0$ such that 
\begin{align*}
	\mu_{\varepsilon}\ge -\frac{\Upsilon_{\infty}(a)}{a^{2}}\;\; \text{for all}\;\; \varepsilon \in (0,\varepsilon_1).
\end{align*}
Combining this with $(V_1)$ and \eqref{222}, one has for all $\varepsilon \in (0,\varepsilon_1)$,
\begin{align*}
	\frac{1}{2}\int_{{\mathbb{R}}^3} |\nabla v_n|^{2}dx+N(v_n)-|\lambda_3|\int_{{\mathbb{R}}^3}|v_n|^{p}dx-\frac{\Upsilon_{\infty}(a)}{a^{2}} \int_{{\mathbb{R}}^3}|v_n|^{2}dx\le o_n(1).
\end{align*}
Together this with Lemma \ref{N<} gives
\begin{align*}
	\begin{split}
		\min\left\lbrace \frac{1}{2}, -\frac{\Upsilon_{\infty}(a)}{a^{2}} \right\rbrace \|v_n\|^{2}_{H^{1}({\mathbb{R}}^3)}&\le -N(v_n)+|\lambda_3|\int_{{\mathbb{R}}^3}|v_n|^{p}dx+o_n(1)\\
		& \le \Lambda \|v_n\|^{4}_{4}+|\lambda_3|\|v_n\|^{p}_{p}+o_n(1).
	\end{split}
\end{align*}
By $v_n \nrightarrow 0$ in $H^{1}({\mathbb{R}}^3)$, then there exists $\iota>0$ independent of $\varepsilon$ such that $\|v_n\|_{H^{1}({\mathbb{R}}^3)} \ge \iota$. Therefore, we have
\begin{align*}
	\liminf\limits_{n \rightarrow +\infty}(\Lambda \|v_n\|^{4}_{4}+|\lambda_3|\|v_n\|^{p}_{p}) \ge 	\min\left\lbrace \frac{1}{2}, -\frac{\Upsilon_{\infty}(a)}{a^{2}} \right\rbrace \iota^{2}.
\end{align*}
Combining this with Lemma \ref{G-N} and Sobolev embedding inequality theorem, there exists $\tau>0$ independent of $\varepsilon \in (0,\varepsilon_1)$ such that
	\begin{align*}
	\liminf\limits_{n \rightarrow +\infty}\|u_n-u_{\varepsilon}\|^{2}_{2} \ge \tau.
\end{align*}
The proof is finished.
\end{proof}

With the above Lemmas in hand, we can prove the following compactness result.
\begin{lemma}\label{lem 4.4}
	Assume that $(V_1)$-$(V_3)$ hold. For $a \in (0,\bar{a})$ and $\varepsilon \in (0,\varepsilon_1)$, the functional $J_{\varepsilon}$ satisfies the $(PS)_{c}$ condition restricted to $W_{\bar{r}}^{a}$ for $c<\Upsilon_{0}(a)+\xi$, where
	\begin{align*}
		0<\xi \le \min\left\lbrace \Upsilon_{\infty}(a)-\Upsilon_{0}(a),\frac{\tau^{2}}{a}(\Upsilon_{\infty}(a)-\Upsilon_{0}(a)) \right\rbrace
	\end{align*}
	and $ \varepsilon_{1}$ is introduced in Lemma \ref{lem 4.3}.
\end{lemma}
\begin{proof}
	Let $\{u_n\}$ be a $(PS)_{c}$ sequence for $J_{\varepsilon}$ restricted to $W_{\bar{r}}^{a}$ with $c<\Upsilon_{0}(a)+\xi$. Note that $c<\Upsilon_{\infty}(a)<0$. Obviously, $\{u_n\}$ is bounded in $H^{1}({\mathbb{R}}^3)$. It follows that there exists $u_{\varepsilon} \in H^{1}({\mathbb{R}}^3)$ satisfying $u_n \rightharpoonup u_{\varepsilon}$ in $H^{1}({\mathbb{R}}^3)$. Combining this with Lemma \ref{lem 4.2} and $c<\Upsilon_{\infty}(a)$, we have $u_{\varepsilon} \neq 0$. Let $v_n:=u_n-u_{\varepsilon}$. We claim $v_n \rightarrow 0$ in $H^{1}({\mathbb{R}}^3)$. Once this claim is true, we complete the proof. Next, we prove the above claim. Suppose by contradiction that $v_n \nrightarrow 0$ in $H^{1}({\mathbb{R}}^3)$. Then it follows from Lemma \ref{lem 4.3} that there exists $\tau>0$ independent of $\varepsilon \in (0, \varepsilon_{1})$ such that
	\begin{align*}
		\liminf\limits_{n \rightarrow +\infty}\|u_n-u_{\varepsilon}\|^{2}_{2} \ge \tau.
	\end{align*}
We set $\|u_{\varepsilon}\|_{2}=b\in (0,a)$, $\|v_n\|_{2}=d_n \in (0,a)$ and assume that $d_n \rightarrow d>0$ as $n \rightarrow \infty$. Thus, $d^{2}\ge \tau$ and $a^{2}=b^{2}+d^{2}$. Moreover, we have $\|\nabla v_n\|_{2}\le \|\nabla u_n\|_{2}<\bar{r}$ and $\|\nabla u_{\varepsilon}\|_{2} \le \liminf\limits_{n \rightarrow +\infty}\|\nabla u_n\|_{2}<\bar{r}$. It follows that $v_n \in W_{\bar{r}}^{d_n}$ and $u_{\varepsilon} \in W_{\bar{r}}^{b}$. Then as the same proof in \eqref{a}, one has for any $\delta>0$,
\begin{align*}
	\begin{split}
	c+o_n(1)&=J_{\varepsilon}(u_n)=J_{\varepsilon}(v_n)+J_{\varepsilon}(u_\varepsilon)+o_n(1)\\
		&\ge J_{\infty}(v_n)-\delta C+J_{0}(u_\varepsilon)+o_n(1)\\
		&\ge \Upsilon_{\infty}(d_n)-\delta C+\Upsilon_{0}(b)+o_n(1),		
	\end{split}
\end{align*}
where $C>0$ independent of $\delta$, $\varepsilon$, $n$. Combining this with \eqref{b}, we have 
\begin{align*}
	\begin{split}
		c+o_n(1)
		\ge \frac{d_n^{2}}{a^{2}}\Upsilon_{\infty}(a)+\frac{b^{2}}{a^{2}}\Upsilon_{0}(a)-\delta C+o_n(1).
	\end{split}
\end{align*}
Letting $n\rightarrow \infty$, it follows from the arbitrariness of $\delta$ that 
\begin{align*}
	\begin{split}
		c&\ge \frac{d^{2}}{a^{2}}\Upsilon_{\infty}(a)+\frac{b^{2}}{a^{2}}\Upsilon_{0}(a)\\
		&=\Upsilon_{0}(a)+\frac{d^{2}}{a^{2}}[\Upsilon_{\infty}(a)-\Upsilon_{0}(a)]\\
		&\ge \Upsilon_{0}(a)+ \frac{\tau}{a^{2}}[\Upsilon_{\infty}(a)-\Upsilon_{0}(a)],
	\end{split}
\end{align*}
contradicting to $c<\Upsilon_{0}+\xi$. Therefore, $u_n \rightarrow u_{\varepsilon}$ in $H^{1}({\mathbb{R}}^3)$. We complete the proof.
\end{proof}

In what follows, we fix $\xi_0$, $\rho_0$ such that:\\
$\bullet$ $\overline{B_{\xi_0}(c_i)} \cap \overline{B_{\xi_0}(c_j)}= \emptyset$ for $i\neq j$ with $i,j \in \{1,2, \cdots,l\}$ and $c_i$, $c_j$ are defined in $(V_3)$;\\
 $\bullet$ $\bigcup^{l}_{i=1}B_{\xi_0}(c_i) \subset B_{\rho_0}(0)$;\\
 $\bullet$
 $K_{\frac{\xi_0}{2}}=\bigcup^{l}_{i=1}\overline{B_{\frac{\xi_0}{2}}(c_i)}$.
 Let us define the function $Q_{\varepsilon}: H^{1}(\mathbb{R}^{3}) \backslash \{0\} \rightarrow \mathbb{R}^{3}$ by
 \begin{align*}
 	Q_{\varepsilon}(u):=\frac{\int_{{\mathbb{R}}^3}\chi(\varepsilon x)|u^{2}|dx}{\int_{{\mathbb{R}}^3}|u^{2}|dx},
 \end{align*}
where $\chi:\mathbb{R}^3 \rightarrow \mathbb{R}^3$ is given by
\begin{align*}
	\chi(x):=
	\begin{cases}
		x, \quad& \text{if}\; |x|\le \rho_0,\\
		\rho_0 \frac{x}{|x|}, \quad& \text{if}\; |x|>\rho_0.
	\end{cases}
\end{align*}

\begin{lemma}\label{lem 4.5}
	Assume that $(V_1)$-$(V_3)$ hold. For $a \in (0,\bar{a})$, there is $0<\xi_1<\xi$ such that if $u \in W_{\bar{r}}^{a}$ and $J_{\varepsilon}(u)\le \Upsilon_{0}(a)+\xi_1$, then 
	\begin{align*}
		Q_{\varepsilon}(u) \in K_{\frac{\xi_0}{2}}, \quad\text{for all}\;\; \varepsilon \in (0,\varepsilon_1),
	\end{align*}
	where $\varepsilon_1$ and $\xi$ are the same as Lemma \ref{lem 4.5}
\end{lemma}
\begin{proof}
	Suppose by contradiction that there is $\xi_n \rightarrow 0$, $\varepsilon_n \rightarrow 0$ and $\{u_n\} \subset W_{\bar{r}}^{a}$ such that
	\begin{align*}
		J_{\varepsilon_n}(u_n) \le \Upsilon_{0}(a)+\xi_n \; \text{and}\; 	Q_{\varepsilon_n}(u_n) \notin K_{\frac{\xi_0}{2}}.
	\end{align*}
	Together this with $(V_1)$ gives 
	\begin{align*}
		\Upsilon_{0}(a)\le J_{0}(u_n) \le J_{\varepsilon}(u_n) \le \Upsilon_{0}(a)+\xi_n,
	\end{align*}
which implies that
	\begin{align*}
		\Upsilon_{0}(a) \rightarrow J_{0}(u_n)\quad \text{as}\; n \rightarrow +\infty.
	\end{align*}
	Therefore, $\{u_n\} \subset W_{\bar{r}}^{a}$ is a minimizing sequence with respect to $\Upsilon_{0}(a)$. Then it follows from Lemma \ref{compactness} that one has two cases:
	
	(i) $u_n \rightarrow u$ in $H^{1}(\mathbb{R}^{3})$ for some $u \in W_{\bar{r}}^{a}$, or
	
	(ii) there exists $\{y_n\}\subset \mathbb{R}^{3}$ with $|y_n|\rightarrow+\infty$ such that the sequence $v_n(x)=u_n(x+y_n)$ is strongly convergent to a function $v\in W^{a}_{\bar{r}}$.	
	
	If case (i) holds. Then by the Lebesgue domainated convergence theorem, we have
	 \begin{align*}
		Q_{\varepsilon_n}(u_n)=\frac{\int_{{\mathbb{R}}^3}\chi(\varepsilon_n x)|u_n^{2}|dx}{\int_{{\mathbb{R}}^3}|u_n^{2}|dx} \rightarrow \frac{\int_{{\mathbb{R}}^3}\chi(0)|u^{2}|dx}{\int_{{\mathbb{R}}^3}|u^{2}|dx}=0\in K_{\frac{\xi_0}{2}}\; \;\;\text{as}\; n \rightarrow +\infty,
	\end{align*}
	which contradicting to $Q_{\varepsilon_n}(u_n) \notin K_{\frac{\xi_0}{2}}$.
	
		If case (ii) holds. When $|\varepsilon_n y_n|\rightarrow + \infty$, then it follows from $v_n \rightarrow v$ in $H^{1}(\mathbb{R}^{3})$ that
		\begin{align}\label{11111}
			\begin{split}
				J_{\varepsilon_n}(u_n)=	J_{\varepsilon_n}(v_n)&=\frac{1}{2}\int_{{\mathbb{R}}^3}|\nabla v_n|^{2}dx+\int_{{\mathbb{R}}^3}V(\varepsilon_n x+\varepsilon_n y_n)|v_n|^{2}dx+\frac{1}{2}N(v_n)\\
				&-\frac{2|\lambda_3|}{p}\int_{{\mathbb{R}}^3}|v_n|^{p}dx
				\rightarrow J_{\infty}(v),\quad \text{as}\; n \rightarrow +\infty.
			\end{split}
		\end{align}
		Combining this with 	$J_{\varepsilon_n}(u_n) \le \Upsilon_{0}(a)+\xi_n$ and $v \in W_{\bar{r}}^{a}$, one has
		\begin{align*}
			\Upsilon_{\infty}(a) \le J_{\infty}(v) \le \Upsilon_{0}(a),
		\end{align*} 
		which contradicts Corollary \ref{cor 3.8}. On the other hand, when $|\varepsilon_n y_n|\rightarrow y_0$ for some $y_0$ in $\mathbb{R}^{3}$. Then similar to \eqref{11111}, one has $J_{\varepsilon}(u_n) \rightarrow J_{y_0}(v)$ as $n \rightarrow +\infty$. 	This combining with 	$J_{\varepsilon_n}(u_n) \le \Upsilon_{0}(a)+\xi_n$ and $v \in W_{\bar{r}}^{a}$, we have
		\begin{align*}
			\Upsilon_{y_0}(a) \le J_{y_0}(v) \le \Upsilon_{0}(a).
		\end{align*} 
		Together this with Corollary \ref{cor 3.8} and $(V_1)$-$(V_3)$ gives $V(y_0)=V_0$ and $y_0=c_i$ for some $i=1,2,\dots,l$. Therefore, one has
		 \begin{align*}
			Q_{\varepsilon_n}(u_n)=\frac{\int_{{\mathbb{R}}^3}\chi(\varepsilon_n x+\varepsilon_n y_n)|v_n^{2}|dx}{\int_{{\mathbb{R}}^3}|v_n^{2}|dx} \rightarrow \frac{\int_{{\mathbb{R}}^3}\chi(y_0)|v^{2}|dx}{\int_{{\mathbb{R}}^3}|v^{2}|dx}=c_i\in K_{\frac{\xi_0}{2}}\; \;\;\text{as}\; n \rightarrow +\infty,
		\end{align*}
			which contradicts $Q_{\varepsilon_n}(u_n) \notin K_{\frac{\xi_0}{2}}$. We complete the proof.
\end{proof}

To obtain multiple normalized solutions, we define 
\begin{align*}
	N_{\varepsilon}^{i}:=\{u \in W_{\bar{r}}^{a}:|Q_{\varepsilon}(u)-c_i| \le \xi_0\},
\end{align*}

\begin{align*}
	\partial N_{\varepsilon}^{i}:=\{u \in W_{\bar{r}}^{a}:|Q_{\varepsilon}(u)-c_i| = \xi_0\}
\end{align*}
and
\begin{align*}
	\beta_{\varepsilon}^{i}:=\inf\limits_{u \in 	N_{\varepsilon}^{i}}J_{\varepsilon}(u),\qquad \tilde{	\beta_{\varepsilon}^{i}}:=\inf\limits_{u \in \partial	N_{\varepsilon}^{i}}J_{\varepsilon}(u).
\end{align*}
\begin{lemma}\label{lem 4.6}
	Assume that $(V_1)$-$(V_3)$ hold. For $a \in (0,\bar{a})$, then there is $\varepsilon_2 \in (0, \varepsilon_1)$ satisfying
	\begin{align*}
		\beta_{\varepsilon}^{i}<\Upsilon_{0}(a)+\xi_1\quad \text{and} \quad 	\beta_{\varepsilon}^{i}<\tilde{	\beta_{\varepsilon}^{i}}\;\;\; \text{for all}\;\; \varepsilon \in (0,\varepsilon_2),
	\end{align*}
	where $\varepsilon_1$ and $\xi_1$ are the same as Lemma \ref{lem 4.5}.
\end{lemma}
\begin{proof}
	Let $u \in W_{\bar{r}}^{a}$ satisfy $J_0(u)=\Upsilon_{0}(a)$. For $1\le i \le l$, we set 
	\begin{align*}
		\hat{u}_{\varepsilon}^{i}(x):=u\left(x-\frac{c_i}{\varepsilon} \right). 
	\end{align*}
	Obviously, $\hat{u}_{\varepsilon}^{i} \in W_{\bar{r}}^{a}$ for all $\varepsilon>0$ and $1 \le i \le l$. Moreover, it follows from the definition of $Q_{\varepsilon}$ that $Q_{\varepsilon}(\hat{u}_{\varepsilon}^{i})\rightarrow c_i$ as $\varepsilon \rightarrow 0^{+}$. Therefore, $\hat{u}_{\varepsilon}^{i} \in N_{\varepsilon}^{i}$ for $\varepsilon$ small enough. Furthermore, we have
	\begin{align*}
		\begin{split}
			J_{\varepsilon}(\hat{u}_{\varepsilon}^{i})&=\frac{1}{2}\int_{{\mathbb{R}}^3}|\nabla \hat{u}_{\varepsilon}^{i}|^{2}dx+\int_{{\mathbb{R}}^3}V(\varepsilon x)|\hat{u}_{\varepsilon}^{i}|^{2}dx+\frac{1}{2}N(\hat{u}_{\varepsilon}^{i})-\frac{2|\lambda_3|}{p}\int_{{\mathbb{R}}^3}|\hat{u}_{\varepsilon}^{i}|^{p}dx\\
			&=\frac{1}{2}\int_{{\mathbb{R}}^3}|\nabla u|^{2}dx+\int_{{\mathbb{R}}^3}V(\varepsilon x+c_i)|u|^{2}dx+\frac{1}{2}N(u)-\frac{2|\lambda_3|}{p}\int_{{\mathbb{R}}^3}|u|^{p}dx.
		\end{split}
	\end{align*}
	Together this with $(V_3)$ gives
	\begin{align*}
		\lim\limits_{\varepsilon \rightarrow 0^{+}}J_{\varepsilon}(\hat{u}_{\varepsilon}^{i})=J_{c_i}(u)=J_{0}(u)=\Upsilon_{0}(a).
	\end{align*}
	Combining this with $\hat{u}_{\varepsilon}^{i} \in N_{\varepsilon}^{i}$ for $\varepsilon$ small enough, we have there is $\varepsilon_2 \in (0,\varepsilon_1)$ such that 
		\begin{align}\label{22222}
		\beta_{\varepsilon}^{i}<\Upsilon_{0}(a)+\xi_1\quad  \text{for all}\;\; \varepsilon \in (0,\varepsilon_2).
		\end{align}
		
	On the other hand, for any $v \in \partial N_{\varepsilon}^{i}$, one has $v \in W_{\bar{r}}^{a}$ and $|Q_{\varepsilon}(v)-c_i|=\xi_0$. Then, $Q_{\varepsilon}(v) \notin K_{\frac{\xi_0}{2}}$. Together this with Lemma \ref{lem 4.5} gives
	 \begin{align*}
	 	J_{\varepsilon}(v)> \Upsilon_{0}(a)+\xi_1\quad \text{for all}\; \;v\in \partial N_{\varepsilon}^{i}\; \;\text{and}\;\; \varepsilon \in (0,\varepsilon_2),
	 \end{align*}
	 which implies that $\tilde{	\beta_{\varepsilon}^{i}}=\inf\limits_{v \in \partial N_{\varepsilon}^{i}}J_{\varepsilon}(v) \ge \Upsilon_{0}(a)+\xi_1$ for all $\varepsilon \in (0,\varepsilon_2)$. Then it follows from \eqref{22222} that
	 \begin{align*}
	 	\beta_{\varepsilon}^{i}<\tilde{	\beta_{\varepsilon}^{i}}\;\;\; \text{for all}\;\; \varepsilon \in (0,\varepsilon_2).
	 \end{align*}	 
	 We complete the proof.
\end{proof}
\begin{proof}[Proof of Theorem \ref{th1.1}] 
Let $\varepsilon_{3} \in (0,\varepsilon_2)$. For all $i=1,2,\cdots, l$ and $\varepsilon \in (0,\varepsilon_3)$, we can use the Ekeland's variational principle \cite{Ekeland 1974} to find a sequence $\{u_n^{i}\}\subset N_{\varepsilon}^{i} \subset  W_{\bar{r}}^{a}$ such that
\begin{align*}
	J_{\varepsilon}(u_{n}^{i})\rightarrow \beta_{n}^{i}\;\;\;\text{and}\;\; 	J_{\varepsilon}(v)-J_{\varepsilon}(u_{n}^{i}) \ge -\frac{1}{n}\|v-u_n^i\|_{H^{1}(\mathbb{R}^{3})},\;\;\;\forall v \in N_{\varepsilon}^{i}\;\;\text{with}\; \;v\neq u_n^i. 
\end{align*}
In view of Lemma \ref{lem 4.6}, one has $\beta_{\varepsilon}^{i}<\tilde{	\beta_{\varepsilon}^{i}}$. Thus, $u_n^i \in N_{\varepsilon}^{i} \backslash \partial N_{\varepsilon}^{i}$ for $n$ large enough. Next, for $\delta>0$ small enough, we define a map $\Phi:(-\delta,\delta) \rightarrow W_{\bar{r}}^{a}$ by
\begin{align*}
	\Phi(t)=a\frac{u_n^i+tv}{\|u_n^i+tv\|^{2}_{2}}
\end{align*}
belonging to $C^{1}((-\delta,\delta),W_{\bar{r}}^{a})$ and satisfying
\begin{align*}
	\Phi(t) \in N_{\varepsilon}^{i} \backslash \partial N_{\varepsilon}^{i}\;\; \forall t \in (-\delta,\delta),\; \Phi(0)=u_n^i,\;\; \Phi'(0)=v,
\end{align*}
where $v \in T_{u_n^i}W_{\bar{r}}^{a}$.
Then, we have
\begin{align*}
	J_{\varepsilon}(\Phi(t))-J_{\varepsilon}(u_n^i) \ge -\frac{1}{n} \|\Phi(t)-u_n^i\|_{H^{1}(\mathbb{R}^{3})}\;\; \forall t \in (-\delta,0)\cup (0,\delta),
\end{align*}
which concludes that 
\begin{align*}
	\begin{split}
		\frac{	J_{\varepsilon}(\Phi(t))-	J_{\varepsilon}(\Phi(0))}{t}&=\frac{	J_{\varepsilon}(\Phi(t))-	J_{\varepsilon}(u_n^i)}{t}\ge -\frac{1}{n}\left\|\frac{\Phi(t)-u_n^i}{t}\right\|_{H^{1}(\mathbb{R}^{3})}\\
		&=-\frac{1}{n}\left\|\frac{\Phi(t)-\Phi(0)}{t}\right\|_{H^{1}(\mathbb{R}^{3})}.
	\end{split}
\end{align*}
By $J_{\varepsilon} \in C^{1}(H^{1}(\mathbb{R}^{3},\mathbb{R}))$ and letting $t \rightarrow 0^{+}$, one has
\begin{align*}
	\left\langle J'_{\varepsilon}(u_n^i),v \right\rangle \ge -\frac{1}{n}\|v\|_{H^{1}(\mathbb{R}^{3})}.
\end{align*}
Replacing $v$ by $-v$, we infer that
\begin{align*}
	\sup \{|\left\langle J'_{\varepsilon}(u_n^i),v \right\rangle|:\|v\|_{H^{1}(\mathbb{R}^{3})}=1 \} \le \frac{1}{n},
\end{align*}
which implies that 
\begin{align*}
	J_{\varepsilon}(u_{n}^{i})\rightarrow \beta_{n}^{i}\;\;\;\text{and}\;\; \|J_{\varepsilon}|^{'}_{W_{\bar{r}}^{a}}(u_{n}^{i})\|_{H^{-1}(\mathbb{R}^{3})}\rightarrow 0\;\;\;\text{as}\; n \rightarrow +\infty.
\end{align*}
Then, $\{u_{n}^{i}\}$ is a $(PS)_{\beta_{n}^{i}}$ for $J_{\varepsilon}$ restricted to $W_{\bar{r}}^{a}$. Combining this with Lemmas \ref{lem 4.4} and \ref{lem 4.6}, we have there is $u^{i} \in H^{1}(\mathbb{R}^{3})$ such that $u_n^{i} \rightarrow u^{i}$ in $H^{1}(\mathbb{R}^{3})$. Thus, 
\begin{align*}
	u^{i} \in N_{\varepsilon}^{i},\quad J_{\varepsilon}(u^{i})=\beta_{\varepsilon}^{i}\;\; \text{and}\;\; J_{\varepsilon}|^{'}_{W_{\bar{r}}^{a}}(u^{i})=0.
\end{align*} 
Moreover, it follows from the definition of $Q_{\varepsilon}(u)$ that
\begin{align*}
	Q_{\varepsilon}(u^{i}) \in \overline{B_{\xi_0}(c_i)},\;\;\;	Q_{\varepsilon}(u^{j}) \in \overline{B_{\xi_0}(c_j)}\;\; \text{and}\;\; \overline{B_{\xi_0}(c_i)} \cap \overline{B_{\xi_0}(c_j)}=\emptyset\;\;\; \text{for}\; i \neq j.
\end{align*}
Then we have $u^{i} \neq u^{j}$ for $i \neq j$, where $1\le i, j \le l$. Therefore, $J_{\varepsilon}|_{S(a)}$ has at least $l$ nontrivial critical points for all $\varepsilon \in (0,\varepsilon_3)$. Then there exists a Lagrange multiplier $\mu_i \in \mathbb{R}$ such that
\begin{align}\label{333}
	\begin{split}
		-\mu_{i} a^{2}&=\frac{1}{2}\int_{{\mathbb{R}}^3}|\nabla u_i|^{2}dx+\int_{{\mathbb{R}}^3}V(\varepsilon x)|u_i|^{2}dx+N(u_i)-|\lambda_3|\int_{{\mathbb{R}}^3}|u_i|^{p}dx\\
		&=\beta_{\varepsilon}^{i}+\frac{1}{2}N(u_i)-\frac{(p-2)|\lambda_3|}{p}\int_{{\mathbb{R}}^3}|u_i|^{p}dx.
	\end{split}
\end{align}
In addition, it follows from Corollary \ref{cor 3.8}, Lemmas \ref{lem 4.4} and \ref{lem 4.6} that $\beta_{\varepsilon}^{i}<0$. Together this with \eqref{333} gives $\mu_{i}>0$. We complete the proof.
\end{proof}

\end{document}